\documentclass[12pt]{article}

\usepackage[a4paper,includeheadfoot,margin=2.54cm]{geometry}
\usepackage[latin2]{inputenc}
\usepackage[english]{babel}
\usepackage{amsthm}
\usepackage{amsmath}
\usepackage{amssymb}
\usepackage[hyphens]{url}

\usepackage{xcolor}
\theoremstyle{plain}
\newtheorem{theorem}{Theorem}
\newtheorem{lemma}{Lemma}
\newtheorem{proposition}[lemma]{Proposition}
\newtheorem{corollary}[lemma]{Corollary}
\newtheorem{fact}[lemma]{Fact}
\newtheorem{example}{Example}[section]
\theoremstyle{definition}

\newtheorem{tab}{Table}
\newtheorem{remark}{Remark}

\newcommand{\fn}{\mathbb{F}_{2^n}}
\newcommand{\f}{\mathbb{F}_{2}}
\newcommand{\F}{\mathbb{F}}
\newcommand{\V}{\mathbb{V}}
\renewcommand{\to}{\rightarrow}
\newcommand{\Tnm}{{\rm Tr^n_m}}
\newcommand{\T}{{\rm Tr}}

\DeclareMathOperator{\Tr}{Tr}

\begin{document}
\title{On functions with the maximal number of bent components}
\author{Nurdag\"{u}l Anbar$^{1}$, Tekg\"{u}l Kalayc\i$^{1}$, Wilfried Meidl$^{2}$, L\'aszl\'o M\'erai$^{2}$
\vspace{0.4cm} \\
\small $^1$Sabanc{\i} University,\\
\small MDBF, Orhanl\i, Tuzla, 34956 \. Istanbul, Turkey\\
\small Email: {\tt nurdagulanbar2@gmail.com}\\
\small Email: {\tt tekgulkalayci@sabanciuniv.edu}\\
\small $^2$Johann Radon Institute for Computational and Applied Mathematics,\\
\small Austrian Academy of Sciences, Altenbergerstrasse 69, 4040-Linz, Austria\\
\small Email: {\tt meidlwilfried@gmail.com} \\
\small Email: {\tt laszlo.merai@oeaw.ac.at}
 }

\maketitle

\begin{abstract}
A function $F:\F_2^n\rightarrow \F_2^n$, $n=2m$, can have at most $2^n-2^m$ bent component functions.
Trivial examples are obtained as {$F(x) = (f_1(x),\ldots,f_m(x),a_1(x),\ldots,\allowbreak a_m(x))$}, where $\tilde{F}(x)=(f_1(x),\ldots,f_m(x))$ is a vectorial bent function from $\F_2^n$ to $\F_2^m$, and $a_i$, $1\le i\le m$, are affine 
Boolean functions. A class of nontrivial examples is given in univariate form with the functions 
$F(x) = x^{2^r}\Tnm(\Lambda(x))$, where $\Lambda$ is a linearized permutation of $\F_{2^m}$.
In the first part of this article it is shown that plateaued functions with $2^n-2^m$ bent components 
can have nonlinearity at most $2^{n-1}-2^{\lfloor\frac{n+m}{2}\rfloor}$, a bound which is attained by 
the example $x^{2^r}\Tnm(x)$, $1\le r<m$ (Pott et al. 2018). This partially solves Question 5 in Pott et al. 2018.
We then analyse the  functions of the form $x^{2^r}\Tnm(\Lambda(x))$. We show that for odd $m$,
only $x^{2^r}\Tnm(x)$, $1\le r<m$, has maximal nonlinearity, whereas there are more of them for 
even $m$, of which we present one more infinite class explicitly. In detail, we investigate Walsh spectrum, 
differential spectrum and their relations for the functions $x^{2^r}\Tnm(\Lambda(x))$. Our results indicate 
that this class contains many nontrivial EA-equivalence classes of functions with the maximal number of bent 
components, if $m$ is even, several with maximal possible nonlinearity.
\end{abstract}

\section{Introduction}

For a Boolean function $f$ from an $n$-dimensional vector space $\V_n$ over the prime field $\F_2$ 
into $\F_2$, the {\it Walsh transform} of $f$ is the integer valued function
 \[ \mathcal{W}_f(u) := \sum_{x\in\V_n}(-1)^{f(x)+\langle u, x\rangle_n}, \]
where $\langle , \rangle_n$ denotes an inner product on $\V_n$. The Boolean function $f$ is called a 
{\it bent function} if $\mathcal{W}_f(u)\in \{\pm 2^{n/2}\}$ for all $u\in \V_n$. Clearly, $n$ must then be even.
If $\mathcal{W}_f(u)\in \{\pm 2^{\frac{n+s}{2}},0\}$ for all $u\in\V_n$, where $s$ is some fixed integer with 
$s\equiv n\bmod 2$, then $f$ is called {\it plateaued}, or more precisely, {\it $s$-plateaued}.
In particular, a bent function is a $0$-plateaued function, $1$-plateaued functions ($n$ odd),
and $2$-plateaued functions ($n$ even) are also called {\it semibent}.

For a {\it vectorial function} $F$ from $\V_n$ to $\V_m$, the {\it component function} $F_\alpha$, 
$\alpha \in \V_m\setminus \{0\}:=\V_m^*$,
is the Boolean function $F_\alpha(x) = \langle \alpha, F(x)\rangle_m$. The Walsh transform of $F$ at $u\in \V_n$
and {$v\in\V_m^*$} is then defined via the component functions as
\[ \mathcal{W}_F(u,v) := \sum_{x\in\V_n}(-1)^{\langle v,F(x)\rangle_m+\langle u, x\rangle_n}. \]

A vectorial function $F:\V_n\rightarrow\V_m$ is called {\it vectorial bent} if all components of $F$ are bent
functions. As it is well known, $m$ can then be at most $n/2$, see \cite{nyb}. We call a vectorial function 
$F$ plateaued, if every component function of $F$ is plateaued (not necessary with the same integer $s$).

Recall that the {\it nonlinearity} $\mathcal{N}_f$ of a Boolean function $f:\V_n\rightarrow\F_2$ is the distance 
of~$f$ to the set of all affine functions, i.e.,
\[ \mathcal{N}_f := \min_{a\in\V_n,c\in\F_2}|\{x\in\V_n\,:\,f(x) \ne \langle a,x\rangle_n + c\}|. \]
The nonlinearity of a vectorial function is the minimal linearity among its component functions.
Similarly as for Boolean functions, the nonlinearity of a vectorial function 
is expressed via its Walsh transform as 
\[ \mathcal{N}_F := 2^{n-1}-\frac{1}{2}\max_{u\in\V_n,\; v\in\V_m^*}|\mathcal{W}_F(u,v)|. \]

For a function $F$ on $\V_n$ and a nonzero element {$a \in \V_n^*$}, let $D_aF(x) = F(x+a)+F(x)$
denote the {\it first-order derivative} of $F$ at $a$. For elements $a \in \V_n^*$, 
$b \in \V_n$, we define
$$
\delta_F(a,b):=|\{ x \in \V_n \mid f(x+a)+f(x)=b \}|.
$$
Clearly, $\delta_F(a,b)$ must be even (since, if $x_0$ is a solution of the above equation, so is $x_0+a$). The \emph{differential uniformity} of $F$ is defined as
$$
\delta_F:=\max_{a \in \V_n^*,\; b \in \V_n} \delta_F(a,b),
$$
and the {\it differential spectrum} of $F$ is the sequence $(\ell_{F,0},\ell_{F,2},\ldots,\ell_{F,\delta_F})$, where 
$\ell_{F,2i}$ is the frequency of $2i$ in the multiset $[\delta_F(a,b) \mid a \in \V_n^*, b \in \V_n]$. 
Recall that the function $F$ is called {\it almost perfect nonlinear} (APN), if its differential uniformity is 
$\delta_F=2$, which is the smallest possible. 

Henceforth we will identify the vector space $\V_n$ with the finite field $\F_{2^n}$, and $n=2m$ is even.
As inner product we may choose $\langle x,y\rangle_n = \T_n(xy)$, where $\T_n(x)$ denotes the absolute 
trace of $x\in\F_{2^n}$.
 
Recall that for a vectorial bent function $F:\F_{2^n}\rightarrow\F_{2^m}$, $m$ can be at most $n/2$.
In particular, for a function on $\F_{2^n}$, not all components can be bent.
In \cite{ppmb}, research on functions on $\F_{2^n}$ with the maximal possible number of bent component
functions is initiated. One of the main results is the following:
\begin{fact}\cite{ppmb}
A function $F$ on $\F_{2^n}$, $n=2m$, can have at most $2^n-2^m$ bent components. If $F$ attains
this maximal number of bent components, then the set
\begin{equation}
\label{setC} 
\mathcal{C} = \{\alpha\in\F_{2^n}\,:\,F_\alpha\;\mbox{is non-bent}\} 
\end{equation}
forms an $m$-dimensional subspace of $\F_{2^n}$. The $2^m$ non-bent component functions hence 
form a vectorial function, which in \cite{ppmb} is called the {\it bent complement} of $F$.
\end{fact}
Note that every vectorial bent function from $\F_{2^n}$ to $\F_{2^m}$ trivially attains this bound if it is seen 
as a function on $\F_{2^n}$. The set $\mathcal{C}$ in $(\ref{setC})$ is then $\mathcal{C} = \F_{2^m}$,
the 
bent complement is the zero-function. We will call the function with the zero-function and more general (after
applying EA-equivalence) an affine function as bent complement, a trivial example of a function with the
maximal number of bent components.
\begin{remark}
\label{muva}
The effect of adding affine functions can also be visualized well in multivariate form. Let 
$F(x) = (f_1(x),\ldots,f_m(x))$ be a vectorial bent function from $\F_2^n$ to $\F_2^m$, $n=2m$,
and let $a_1(x),\ldots,a_m(x)$ be affine functions. Then the function $F_1(x) = (f_1(x),\ldots,f_m(x),a_1(x))$
from $\F_2^n$ to $\F_2^{m+1}$ has $a_1$ as the only nontrivial non-bent component function, the non-bent 
component functions of {$F_k(x) = (f_1(x),\ldots,f_m(x),a_1(x),\allowbreak \ldots,a_k(x))$} for some $1\le k\le m$, are the
linear combinations of $a_1,\ldots,a_k$. In particular, $F_m(x) = (f_1(x),\ldots,f_m(x),a_1(x),\ldots,a_m(x))$
is a trivial example of a function on $\F_2^n$ with the maximal possible number of bent components, its
bent complement is the function $C(x) = (a_1(x),\ldots,a_m(x))$ from $\F_2^n$ to $\F_2^m$.
By these observations, functions with the maximal number of bent components, without any (nontrivial) affine 
component function are {the most} interesting. 
\end{remark}
As a nontrivial example of a function on $\F_{2^n}$ with the maximal number of bent components, in \cite{ppmb},
the function $F(x) = x^{2^r}\Tnm(x)$, $0 < r < m$, is presented, 
where $\Tnm$ denotes the relative trace mapping from $\F_{2^n}$ to $\F_{2^m}$, for which the component function
$F_\gamma(x) = \T_n(\gamma F(x))$ is bent if and only if $\gamma \not\in \F_{2^m}$.  
In \cite{mztz}, more general, functions of the form $F(x) = x^{2^r}\Tnm(\Lambda(x))$ are considered, where
$\Lambda(x) = x + \sum_{j=1}^\sigma\alpha_jx^{2^tj}$ is a linearized polynomial in $\F_{2^m}[x]$.
Recently, in \cite{akm1}, functions $F$ of such a form with the maximal number of bent components are completely
characterized as follows (see also \cite{zpkll} for a similar result): 
\begin{fact}\cite{akm1}
Let $F(x) = x^{2^r}\Tnm(\Lambda(x))$ for some integer $r\ge 0$ and a linearized polynomial $\Lambda$ with coefficients
in $\F_{2^m}$. The component function $F_\gamma(x) = \T_n(\gamma x^{2^r}\Tnm(\Lambda(x)))$ is bent for 
all $\gamma\not\in\F_{2^m}$, if and only if $\Lambda$ is a permutation of $\F_{2^m}$ (and hence of $\F_{2^n}$),
and consequently $F$ has the maximal number of bent components. 
\end{fact}
Clearly, having the maximal number of bent components is an EA-equivalence invariant (and as shown in \cite{mztz}
it is even invariant under CCZ-equivalence). Hence, up to EA-equivalence, we can cover all functions of the form
$F(x) = x^{2^r}\Tnm(\Lambda(x))$ with one fixed integer $r$. We though decide not to keep $r$ fixed, in order to 
have more freedom in finding examples. For the same reason we may assume that $\Lambda$ is normalized, i.e.,
the coefficient for the monomial in $\Lambda$ with largest degree is $1$.

In this article we first address a question in \cite{ppmb} on the maximal nonlinearity of functions on $\F_{2^n}$ 
with the maximal number of bent components. We completely solve this question in Section \ref{sec2} for the class 
of plateaued functions. In Section \ref{sec3}, we analyse in detail Walsh spectrum and differential spectrum for
functions $F$ with the maximal number of bent components of the form $F(x) = x^{2^r}\Tnm(\Lambda(x))$.
We are particularly interested in those functions that attain the maximal possible nonlinearity.
In Section \ref{sec4}, we present some numerical results which also indicate that functions 
with the maximal number of bent components of the form $F(x) = x^{2^r}\Tnm(\Lambda(x))$ separate in many
CCZ-equivalence classes.

\section{Best nonlinearity for functions with maximal number of bent components}
\label{sec2}

We start this section with some useful identities involving the Walsh transform. 

Let $f$ be a Boolean function $f:\fn\to\f$ and $\mathcal{W}_f(u)=\sum_{x\in\fn}(-1)^{f(x)+\T_n(ux)}$ 
its Walsh transform. The \emph{sum-of-square indicator} of $f$ is then defined as
\[ \nu(f)=\sum_{a\in\fn}\mathcal{W}_{D_a f}(0)^2=2^{-n}\sum_{a\in\fn}\mathcal{W}_{f}(a)^4. \]
We will use the following two {lemmas} from \cite{bccl}.
\begin{lemma}\label{lemma:sum-of-square}
Any Boolean function $f:\fn\to\f$ satisfies $\nu(f)\leq 2^n \max_{u\in\fn}(\mathcal{W}_f(u))^2$. 
Equality occurs if and only if $f$ is plateaued, i.e.,
\[ \max_{u\in\fn}|\mathcal{W}_f(u)|=2^{\frac{n+k}{2}} \quad \text{and } \nu(f)=2^{2n+k} 
\quad \text{with } 0\leq k \leq \frac{n}{2}, \quad n\equiv k \mod 2. \]
\end{lemma}

\begin{lemma}\label{theotherone} 
Let $F:\fn \to \fn$ be a vectorial function. Then
\begin{align*}
 \sum_{v\in\fn^*}\nu(F_v)\geq (2^n-1)2^{2n+1}.
\end{align*} 
Equality occurs if and only if $F$ is APN.
\end{lemma}
For a plateaued vectorial function $F:\F_{2^n}\rightarrow\F_{2^n}$ let
\[ S_i=\left\{v\in\fn^*: \ \max_{u\in\fn}\left|\mathcal{W}_{F_v}(u) \right|= 2^{\frac{n+i}{2}}\right\} \]
be the subset of $\F_{2^n}$ that corresponds to the set of $i$-plateaued component functions.
\begin{proposition}\label{pro:S_i}
Let $F:\fn\to\fn$ be a plateaued vectorial function with nonlinearity $2^{n-1}-\frac{1}{2}2^{\frac{n+k}{2}}$. 
Then we have
\begin{align}\label{eq:Si}
 \sum_{i=0}^{k-1}\left(2^k-2^i \right)|S_i|\leq (2^{k}-2)(2^n-1)
\end{align}
with equality if and only if $F$ is APN.
\end{proposition}
\begin{proof}
By the assumption on the nonlinearity and the plateauedness of $F$ we have $|S_k|=2^n-1-\sum_{i=0}^{k-1} |S_i|>0$. By Lemma \ref{lemma:sum-of-square}
and Lemma \ref{theotherone} we have
\[
 (2^n-1)2^{2n+1} \leq \sum_{j=0}^k \sum_{v\in S_j}\nu(F_v)=\sum_{j=0}^k  |S_j| 2^{2n+j}.
\]
Dividing by $2^{2n}$ and substituting $|S_k|=2^n-1-\sum_{i=0}^{k-1} |S_i|$, we get
\[
 2(2^n-1)\leq \sum_{j=0}^{k-1} 2^{j} |S_j| +(2^n-1)2^k - 2^{k} \sum_{j=0}^{k-1}  |S_j|,
\]
i.e.,
\[
  (2^k-2)(2^n-1)\geq \sum_{j=0}^{k-1} (2^k-2^{j}) |S_j|,
\]
with equality if and only if $F$ is APN.
%
\end{proof}

Before we present the main results of this section on the maximal nonlinearity for plateaued functions
with the largest possible number of bent components, we infer from Proposition \ref{pro:S_i}
lower and upper bounds on the number of bent components for plateaued APN functions $F$ with given nonlinearity.
The upper bound holds independently from $F$ being APN, and will then give rise to Theorem \ref{nobo} below.
%
%
\begin{corollary}\label{lem:bound}
Let $n$ be an even integer and $F:\fn\to\fn$ be a plateaued APN function with nonlinearity 
$2^{n-1}-\frac{1}{2}2^{\frac{n+k}{2}}$.
Then the number of bent components of $F$ satisfies the following inequalities.
\begin{align}\label{eq:upper}
\frac{2}{3}(2^n -1)+\frac{1}{3}(2^k-2^2)   \leq |S_0|\leq (2^n-1)\left( 1-\frac{1}{2^k-1}\right).
\end{align}
The upper bound holds independently from $f$ being APN.
\end{corollary}
\begin{proof}
For any plateaued function on $\F_{2^n}$ with nonlinearity $2^{n-1}-\frac{1}{2}2^{\frac{n+k}{2}}$, 
by Equation~\eqref{eq:Si}, we have 
\begin{align*}
|S_0|(2^k-1) \le (2^{k}-2)(2^n-1)-\sum_{i=2}^{k-2}\left(2^k-2^i \right)|S_i|.
\end{align*}
In particular, we have $|S_0|(2^k-1)\leq (2^{k}-2)(2^n-1)$, which gives the desired upper bound. 

Observe that $|S_k|\geq 1 $ implies that 
\begin{align}\label{eq:1}
\sum_{i=2}^{k-2}|S_i|\leq 2^n -2-|S_0|.
\end{align}
Also note that for any $i=2,\ldots, k-2$ we have $2^k-2^{i}\leq 2^k-2^{2}$, and hence
\begin{align}\label{eq:2}
\sum_{i=2}^{k-2}\left(2^k-2^i \right)|S_i|\leq (2^k-2^{2})\sum_{i=2}^{k-2}|S_i|.
\end{align}
Then, by Equations \eqref{eq:1} and \eqref{eq:2}, we obtain
\begin{align*}
\sum_{i=2}^{k-2}\left(2^k-2^i \right)|S_i|\leq (2^k-2^{2})(2^n -2-|S_0|).
\end{align*}
Using that for an APN function we have
\begin{align*}
|S_0|(2^k-1)=(2^{k}-2)(2^n-1)-\sum_{i=2}^{k-2}\left(2^k-2^i \right)|S_i|.
\end{align*}
We conclude that
\begin{align*}
|S_0|(2^k-1)&\geq (2^{k}-2)(2^n-1)-(2^k-2^{2})(2^n -2-|S_0|)\\
&=(2^{k}-2)(2^n-1)-(2^k-2^{2})(2^n -2)+|S_0|(2^k-2^{2}),
\end{align*}
and hence we have
\begin{align*}
3|S_0|\geq 2^{n+1}+2^k-6,
\end{align*}
which gives the desired lower bound in Equation \eqref{eq:upper}.
\end{proof}

%

In Question 5 in \cite{ppmb}, the authors propose the problem of finding the best nonlinearity for functions on $\F_{2^n}$ with the maximal number of bent components. In the following theorem we give a tight upper bound on the nonlinearity
for plateaued functions with the maximal number of bent components.
We remark that having the maximal possible number of bent components is a CCZ-equivalence invariant, see
\cite[Theorem 1]{mztz}, hence the bound in the theorem below applies to any function with the maximal number
of bent components that is CCZ-equivalent to a plateaued function.
%
\begin{theorem}
\label{nobo}
Let $F:\F_{2^n}\rightarrow\F_{2^n}$, $n = 2m$, be a plateaued vectorial function with the maximal number 
$2^n-2^m$ of bent components. Then the nonlinearity of $F$ is at most $2^{n-1}-2^{\lfloor\frac{n+m}{2}\rfloor}$.
\end{theorem}
{\it Proof.}
Let $F$ be a plateaued function with nonlinearity $2^{n-1}-\frac{1}{2}2^{\frac{n+k}{2}}$,
i.e., $F$ has at least one $k$-plateaued component function, but no $s$-plateaued component function 
for any $s > k$. If $F$ has the maximal possible number of bent components, i.e., $|S_0| = 2^n-2^{n/2}$, 
then by Corollary \ref{lem:bound} we have
\[ 2^n-2^{n/2} \le (2^n-1)\left(1 - \frac{1}{2^k-1}\right). \]
Solving this equation, we obtain $2^k \ge 2^{n/2}+2$, which implies that $k\ge n/2+1$ if $n/2$ is odd and $k\ge n/2+2$ if $n/2$ is even.
\hfill$\Box$\\[.3em]
%
%

By \cite[Theorem 7]{ppmb}, if $\gcd(r,m) = 1$, then the bent complement of $F(x) = x^{2^r}\Tnm(x)$
has only $(m+1)$-plateaued components if $m$ is odd, and $m$-plateaued and $(m+2)$-plateaued components if $m$ is even. The bound in Theorem \ref{nobo} is hence tight. As one can see from the proof of
\cite[Theorem 7]{ppmb}, in the latter case, the bent complement of $x^{2^r}\Tnm(x)$ has exactly $2(2^m-1)/3$
components that are $m$-plateaued and $(2^m-1)/3$ components that are $(m+2)$-plateaued. 
\begin{remark}
\label{Sm<}
The upper bound in Corollary \ref{lem:bound} also reveals the known fact that a plateaued function on
$\F_{2^n}$, $n$ even, with only bent and semibent components, hence with maximal possible nonlinearity,
can have at most $2(2^n-1)/3$ bent components. 
\end{remark}

%

\section{Walsh and differential spectrum of $x^{2^r}\Tnm(\Lambda(x))$}
\label{sec3}

In this section, we analyse Walsh spectrum and differential spectrum for the functions on $\F_{2^n}$, $n=2m$,
with the maximal number of bent components of the form $F(x) = x^{2^r}\Tnm(\Lambda(x))$.
According to Fact 2, $\Lambda$ will always be a linearized permutation over $\F_{2^m}$.

We first attempt to generalize the analysis of the Walsh spectrum of $x^{2^r}\Tnm(x)$, $\gcd(r,m) = 1$,
in \cite{ppmb}, to arbitrary functions of the form $F(x) = x^{2^r}\Tnm(\Lambda(x))$.
From Fact 2 we know that $F_\gamma$ is bent if and 
only if $\gamma \not\in\F_{2^m}$. Therefore we are interested in the Walsh spectrum of the bent complement
$\{F_\gamma\,:\,\gamma\in\F_{2^m}\}$.
\begin{theorem}
\label{Wosp}
For a linearized permutation $\Lambda \in\F_{2^m}[x]$, let $F:\F_{2^n}\rightarrow\F_{2^n}$, $F(x) = x^{2^r}\Tnm(\Lambda(x))$, and let $H:\F_{2^m}\rightarrow\F_{2^m}$
be given as $H(x) = x^{2^r}\Lambda(x)$. For a nonzero element $\alpha\in\F_{2^m}$ we have $\mathcal{W}_F(\alpha,\beta) = 0$ if $\beta\not\in\F_{2^m}$, and 
\[ \mathcal{W}_F(\alpha,\beta) = 2^m\mathcal{W}_H(\alpha,\beta) \]
if $\beta\in\F_{2^m}$.
\end{theorem}
{\it Proof.}
We write $x\in \F_{2^n}$ uniquely as $x = y+u$, where $y\in\F_{2^m}$ and $u\in U$, a fixed complement $U$ of $\F_{2^m}$. Then for a nonzero $\alpha\in\F_{2^m}$ and
$\beta\in\F_{2^n}$ we have
\begin{align*}
\mathcal{W}_F(\alpha,\beta) & = \sum_{x\in\F_{2^n}}(-1)^{\T_n(\alpha x^{2^r}\Tnm(\Lambda(x))) + \T_n(\beta x)} \\
& = \sum_{u\in U}\sum_{y\in\F_{2^m}}(-1)^{\T_n(\alpha(y^{2^r}+u^{2^r})\Tnm(\Lambda(y)+\Lambda(u))) + \T_n(\beta(y+u))} \\
& = \sum_{u\in U}(-1)^{\T_n(\alpha u^{2^r}\Tnm(\Lambda(u))) + \T_n(\beta u)}\sum_{y\in\F_{2^m}}(-1)^{\T_n(\alpha y^{2^r}\Tnm(\Lambda(u)) + \T_n(\beta y)}.
\end{align*}
Observing that 
$$
{
\T_n(\alpha y^{2^r}\Tnm(\Lambda(u)) = \T_m(\Tnm(\alpha y^{2^r}\Tnm(\Lambda(u)))) = \T_m(\Tnm(\Lambda(u))\Tnm(\alpha y^{2^r})) = 0
}
$$
(since $\alpha, y\in\F_{2^m}$), we obtain
\begin{equation}
\label{ppmb12}
\mathcal{W}_F(\alpha,\beta) = \sum_{u\in U}(-1)^{\T_n(\alpha u^{2^r}\Tnm(\Lambda(u))) + \T_n(\beta u)}\sum_{y\in\F_{2^m}}(-1)^{\T_n(\beta y)}.
\end{equation}
If $\beta\in\F_{2^n}\setminus\F_{2^m}$, then the inner sum in $(\ref{ppmb12})$ vanishes. Consequently, $\mathcal{W}_F(\alpha,\beta) = 0$.
If $\beta \in\F_{2^m}$, then we obtain $2^m$ for the inner sum in $(\ref{ppmb12})$. Hence
\[ \mathcal{W}_F(\alpha,\beta) = 2^m\sum_{u\in U}(-1)^{\T_n(\alpha u^{2^r}\Tnm(\Lambda(u))) + \T_n(\beta u)}. \]
Let
\[ \Lambda(x) = \sum_{j=0}^\sigma\alpha_jx^{2^{t_j}} \]
(w.l.o.g. we can assume $t_0 = 0$ and $\alpha_0 =1$), then
\[ u^{2^r}\Tnm(\Lambda(u)) = u^{2^r}\sum_{j=0}^\sigma \alpha_j(u^{2^{t_j}} + (u^{2^m})^{2^{t_j}}). \]
As $u\rightarrow u+u^{2^m}$ is a bijection from $U$ to $\F_{2^m}$, we can substitute $u+u^{2^m}$ by $z$, and obtain
\begin{align*}
\mathcal{W}_F(\alpha,\beta) & = 2^m\sum_{u\in U}(-1)^{\T_n(\alpha u^{2^r}\sum_{j=0}^\sigma \alpha_j(u^{2^{t_j}} + (u^{2^m})^{2^{t_j}})) + \T_m(\beta (u+u^{2^m}))} = \\
& = 2^m\sum_{z\in\F_{2^m}}(-1)^{\T_m(\alpha \Tnm(u^{2^r}\Lambda(z))) + \T_m(\beta z)} = 2^m\sum_{z\in\F_{2m}}(-1)^{\T_m(\alpha\Lambda(z)z^{2^r}) + \T_m(\beta z)} \\
& = 2^m\mathcal{W}_H(\alpha,\beta).
\end{align*}
\hfill$\Box$\\[.5em]
In \cite[Theorem 6]{ppmb} it is shown that for $F(x) = x^{2^r}\Tnm(x)$ we have $\delta_F(a,b) = 2^m$ if 
$a\in\F_{2^m}^*$, and $\delta_F(a,b) = 2^{\gcd(r,m)}$ if $a\not\in\F_{2^m}$. We attempt to analyse the differential spectrum more general for arbitrary  functions 
$F(x) = x^{2^r}\Tnm(\Lambda(x))$, where $\Lambda\in\F_{2^m}[x]$ is a linearized permutation.
We will see that, as for the Wash spectrum, the differential spectrum of $F$ is completely determined 
by the properties of the quadratic function $H(x) = x^{2^r}\Lambda(x)$ on~$\F_{2^m}$.  
\begin{theorem}
\label{delta}
For a linearized permutation $\Lambda\in\F_{2^m}[x]$, $m=n/2$, let $F(x) = x^{2^r}\Tnm(\Lambda(x))$,
and let $H(x) = x^{2^r}\Lambda(x)$. For the differential spectrum of $F$ we have
\[ \delta_F(a,b) \in \left\{\begin{array}{l@{\quad:\quad}l}
 \{0,2^m\} & a\in\F_{2^m}^*, \\
\{0,2^s\} & a\in\F_{2^n}\setminus\F_{2^m}, 
\end{array}\right. \]
where $s$ is the dimension of the solution space (in $\F_{2^m}$) of $z^{2^r}\Lambda(c) + c^{2^r}\Lambda(z)=0$,
$c=\Tnm(a)$, i.e., $\delta_F(a,b) =\delta_H(\Tnm(a),d) \in \{0,2^s\}$ (with $d\in\F_{2^m}$).
\end{theorem}
{\it Proof.}
Since $F$ is quadratic, $\delta_F(a,b) \in \{0,2^s\}$, where $s$ is the dimension of the solution space of 
 \begin{align*}
 F(x+a) + F(x) + F(a) = x^{2^r}\Tnm(\Lambda(a)) + a^{2^r}\Tnm(\Lambda(x))=0.
 \end{align*}
For $a\in \mathbb{F}_{2^m}$,  this implies $a^{2^r}\Tnm(\Lambda(x))=0$, which holds if and only if 
$x\in\F_{2^m}$. Now suppose that $a\in \mathbb{F}_{2^n} \setminus \mathbb{F}_{2^m}$.  
Since $\Tnm(\Lambda(a)), \Tnm(\Lambda(x)) \in \mathbb{F}_{2^m}$, the equality $x^{2^r}\Tnm(\Lambda(a)) + a^{2^r}\Tnm(\Lambda(x))=0$ implies that $x \in a\mathbb{F}_{2^m}$, i.e., $x=ay$ for some $y\in \mathbb{F}_{2^m}$. Therefore,
 \begin{align}\label{eq:deriF11}
 &x^{2^r}\Tnm(\Lambda(a)) + a^{2^r}\Tnm(\Lambda(x))=
 a^{2^r}y^{2^r}\Tnm(\Lambda(a)) + a^{2^r}\Tnm(\Lambda(ay))=0
 \end{align}
 if and only if $y^{2^r}\Tnm(\Lambda(a)) + \Tnm(\Lambda(ay))=0$. Since $\Lambda$ has coefficients in $\mathbb{F}_{2^m}$, we have $\Tnm(\Lambda(a))=\Lambda(\Tnm(a))$ for any $a\in \mathbb{F}_{2^n}$. Set $c=\Tnm(a)$. Note that $c\neq 0$ since $a\not\in \mathbb{F}_{2^m}$.  That is, Equation \eqref{eq:deriF11} holds if and only if 
 \begin{align}\label{eq:deriF222}
 y^{2^r}\Lambda(\Tnm(a)) + \Lambda(\Tnm(ay))= y^{2^r}\Lambda(\Tnm(a)) + \Lambda(y\Tnm(a))
 = y^{2^r}\Lambda(c) + \Lambda(cy)=0 .
 \end{align}
 Set $z=cy$. Then Equation \eqref{eq:deriF222} is equivalent to $(z^{2^r}/c^{2^r})\Lambda(c) + \Lambda(z)=0$. Therefore, Equation \eqref{eq:deriF11} holds if and only if $z^{2^r}\Lambda(c) + c^{2^r}\Lambda(z)=0$. Hence, $\delta_F(a,b) \in \{0,  2^{s}\}$ for $a\in \mathbb{F}_{2^n} \setminus \mathbb{F}_{2^m} $, where $s$ is the dimension of the solution space of $z^{2^r}\Lambda(\Tnm(a)) + \Tnm(a)^{2^r}\Lambda(z)=0$.
\hfill$\Box$\\[.5em]
\begin{example}
The results on the Walsh spectrum of $x^{2^r}\Tnm(x)$ in \cite{ppmb} for the case that $\gcd(r,m) = 1$
follow from Theorem \ref{Wosp} with the known fact that in this case $H(x) = x^{2^r+1}$ is almost bent 
if $m$ is odd, i.e., all components of $H$ are semibent, and only has bent and semibent components if $m$
is even. 
Similarly, we recover \cite[Theorem 6]{ppmb}, where it is shown that for $F(x) = x^{2^r}\Tnm(x)$ we have 
$\delta_F(a,b)\in \{0,2^{\gcd(r,m)}\}$ for all $a\in\F_{2^n}\setminus\F_{2^m}$ from Theorem \ref{delta}
with the known differential spectrum of $x^{2^r+1}$.
\end{example}
As the Walsh spectrum of $x^{2^r+1}$ is also known for the case that $\gcd(m,r) > 1$ (see for instance 
\cite{coulter}), applying Theorem \ref{Wosp}, we can derive the complete Walsh spectra for the functions
$x^{2^r}\Tnm(x)$, which complements the results in \cite{ppmb} for the special case that $\gcd(m,r) = 1$. 
\begin{corollary}
Let $n=2m$, let $F(x) = x^{2^r}\Tnm(x)$ and suppose that $\gcd(m,r) = d$.
\begin{itemize}
\item[(i)] If $m/d$ is odd, then for $\alpha \in\F_{2^m}^*$ and $\beta\in\F_{2^n}$ we have
$$
{
W_F(\alpha,\beta) \in\{0,\pm 2^{\frac{n+m+d}{2}}\}.
}
$$
\item[(ii)] If $m/d$ is even, then for $\alpha \in\F_{2^m}^*$ and $\beta\in\F_{2^n}$ we have
$$
{
W_F(\alpha,\beta) \in\{0,\pm 2^{\frac{n+m}{2}},\pm 2^{\frac{n+m+2d}{2}}\}.
}$$
\end{itemize}
\end{corollary}

Attaining the bound in Theorem \ref{nobo}, the function $x^{2^r}\Tnm(x)$ exhibits best nonlinearity in
both cases, $m$ odd and $m$ even, whenever $\gcd(m,r) = 1$. By a result in \cite{bccl}, for $m$ odd,
it is the only example for a function of the form $x^{2^r}\Tnm(\Lambda(x))$ with maximal nonlinearity
(up to equivalence via a Frobenius automorphism).
\begin{lemma}\cite[Proposition 7]{bccl}
\label{frombccl}
The only APN-function $H$ on a finite field $\F_{2^m}$ of the form $H(x) = x\Lambda(x)$, where 
$\Lambda$ is a linearized polynomial, is $H(x) = \gamma x^{2^r+1}$ with $\gcd(r,m) = 1$, 
$\gamma\in\F_{2^m}^*$.
\end{lemma}
\begin{corollary}
\label{allbest}
For a linearized permutation $\Lambda\in\F_{2^m}[x]$, let $F(x) = x^{2^r}\Tnm(\Lambda(x))$.
\begin{itemize}
\item[(i)] If $m$ is odd, then the following statements are equivalent.
\begin{itemize}
\item[I.] $F$ has the maximal nonlinearity $\mathcal{N}_F = 2^{n-1}-2^{\frac{n+m-1}{2}}$.
\item[II.] $\delta_F(a,b) \in\{0,2\}$ for all $a\in\F_{2^n}\setminus\F_{2^m}$.
\item[III.] $F(x) = x^{2^r}\Tnm(x)$ with $\gcd(r,m) = 1$ (up to equivalence via a Frobenius automorphism 
and the multiplication of a constant $\gamma\in\F_{2^m}^*$).
\end{itemize}
\item[(ii)] If $m$ is even, then $\delta_F(a,b) \in\{0,2\}$ for all $a\in\F_{2^n}\setminus\F_{2^m}$ if and only if
$F(x) = x^{2^r}\Tnm(x)$ with $\gcd(r,m) = 1$ (up to equivalence via a Frobenius automorphism and the 
multiplication of a constant $\gamma\in\F_{2^m}^*$).
\end{itemize}
\end{corollary}
{\it Proof.}
First consider the case that $m$ is odd. By Theorem \ref{Wosp}, $F$ has the maximal possible nonlinearity
if and only if $H(x) = x^{2^r}\Lambda(x)$ has only semibent components, i.e., $H$ is {\it almost bent}.
Since $H$ is quadratic, $H$ is almost bent if and only if $H$ is APN, i.e., $\delta_F(a,b) \in\{0,2\}$ for all 
$a\in\F_{2^n}\setminus\F_{2^m}$ by Theorem \ref{delta}.
Via a Frobenius automorphism, $F$ is equivalent to a function $\tilde{F}(x) = x\Tnm(\tilde{\Lambda}(x))$,
for some linearized permutation $\tilde{\Lambda}$. We know that $\tilde{F}$ has the maximal nonlinearity
if and only if $\tilde{H}(x) = x\tilde{\Lambda}(x)$ is APN, which by Lemma \ref{frombccl} applies if and only
if $H(x) = x^{2^{\tilde{r}}+1}$ for some $\tilde{r}$ with $\gcd(\tilde{r},m) =1$ 
(up to the multiplication with a constant $\gamma\in\F_{2^m}^*$).
Finally $\tilde{F}(x) = x\Tnm(x^{\tilde{r}+1})$ is equivalent to $F(x) = x^{2^r}\Tnm(x)$, where
$r = m-\tilde{r}$ (via the Frobenius automorphism $x \rightarrow x^{2^{m-\tilde{r}}}$).

Note that the argument for the equivalence of II and III is independent from $m$ being odd, which concludes
also the proof for the case (ii) when $m$ is even.
\hfill$\Box$\\[.5em]
By Corollary \ref{allbest}, for $m$ odd there is only one type of functions of the form 
$x^{2^r}\Tnm(\Lambda(x))$ with maximal nonlinearity, which then also has the smallest possible 
values in the differential spectrum. This is different for even $m$, where we have only one type of
functions with smallest possible values in the differential spectrum, we then also have maximal nonlinearity, 
but as also Table \ref{tab1} shows, there are many classes with maximal nonlinearity and 
$\delta_F(a,b) \in\{0,2,4\}$ if $a\in\F_{2^m}\setminus\F_{2^m}$, which can be considered second best.
For such functions, the differential spectrum is completely described by the Walsh spectrum.
%
%
%
\begin{corollary}
\label{4-number}
Let $m$ be even, and suppose that for a linearized permutation $\Lambda$ of $\F_{2^m}$ the function
$F(x) = x^{2^r}\Tnm(\Lambda(x))$ attains the maximal nonlinearity, i.e., for $F$ we have $|S_0| = 2^n-2^m$, 
$|S_m| > 0$, $|S_{m+2}| > 0$, and $|S_k|=0$ if $k\not\in \{0,m,m+2\}$, and suppose that 
$\delta_F(a,b)\in\{0,2\}$ or $\delta_F(a,b)\in\{0,4\}$ for every $a\in\F_{2^n}\setminus\F_{2^m}$.
Then the number of $a\in\F_{2^n}\setminus\F_{2^m}$ for which $\delta_F(a,b)\in\{0,4\}$ is given by
\begin{equation} 
\label{noofa}
2^m(2^m-1) - 3\cdot 2^{m-1}|S_m|. 
\end{equation}
In particular, $\delta_F(a,b) \in\{0,2\}$ for all $a\in\F_{2^n}\setminus\F_{2^m}$ if and only if $|S_m| = 2(2^m-1)/3$,
which applies if and only if $F$ is equivalent to $x^{2^r}\Tnm(x)$ for some $r$ with $\gcd(r,m)=1$.
\end{corollary}
{\it Proof.} 
Let $H:\F_{2^m}\rightarrow\F_{2^m}$, $m$ even, be a function with only bent and semibent components
and suppose that $\delta_H(\alpha,\beta) \in \{0,2,4\}$ for all $(\alpha,\beta) \in\F_{2^m}^*\times\F_{2^m}$.
By Proposition~22 in \cite{cdp}, the number of pairs $(\alpha,\beta) \in\F_{2^m}^*\times\F_{2^m}$ for which
$\delta_H(\alpha,\beta) = 4$ is then fixed as
\[ \ell_{H,4} = 2^{m-2}(2^m-1) - 3\cdot 2^{m-3}|S_0|, \]
where here, $|S_0|$ denotes the number of bent components for the function $H$.
Since in our case $H(x) = x^{2^r}\Lambda(x)$ is quadratic, $\delta_H(\alpha,\tilde{\beta}) = 4$ for some 
$\tilde{\beta}\in\F_{2^m}$, implies that $\delta_H(\alpha,\beta)\in\{0,4\}$ for all $\beta\in\F_{2^m}$, and there 
are exactly $2^{m-2}$ elements $\beta\in\F_{2^m}$ for which $\delta_H(\alpha,\beta) = 4$.
As we have $\Tnm(a) = \alpha$ for exactly $2^m$ elements $a\in\F_{2^n}$, by Theorem \ref{delta} we have 
$4\cdot\ell_{H,4}$ elements $a\in\F_{2^n}\setminus\F_{2^m}$ for which $\delta_F(a,b)\in\{0,4\}$.

Observe that the term in $(\ref{noofa})$ vanishes if and only if $|S_m| = 2(2^m-1)/3$.
(Recall that by Remark \ref{Sm<}, $|S_m|$ is at most $2(2^m-1)/3$.)
With Corollary \ref{allbest} (ii), the last statement then follows.
\hfill$\Box$\\[.5em]

We illustrate our results presenting a further class of functions $F(x) = x^{2^r}\Tnm(\Lambda(x))$ with maximal
nonlinearity (and $\delta_F(a,b)\in\{0,4\}$ for all $a\in\F_{2^n}\setminus\F_{2^m}$).
\begin{proposition}
\label{bino}
Let $m\equiv 2 \mod 4$, let $\beta$ be an element of $\F_{2^m}$ which is not a $(2^{2r}-1)$th power, and
let $\Lambda (x)=x^{2^{3r}}+\beta x^{2^{r}}\in \mathbb{F}_{2^m}[x]$. 
Then $F(x) = x^{2^r}\Tnm(\Lambda(x))$ is a function with the maximal number of bent components, 
largest possible nonlinearity, and $\delta_F(a,b) \in \{0,4\}$ for all $a\in\F_{2^n}\setminus\F_{2^m}$.
More precisely, $F$ has only bent and $(m+2)$-plateaued components.
\end{proposition}
{\it Proof.}
First note that $\Lambda$ is a permutation of $\mathbb{F}_{2^m}$ if and only if 
$\beta \not\in \mathrm{Im}(x^{2^{2r}-1})$. Hence $F$ has the maximal number of bent components.

To show that $\delta_F(a,b) \in \{0,4\}$ for all $a\in\F_{2^n}\setminus\F_{2^m}$, we determine
the differential properties of the function $H(x)=x^{2^{r}}\Lambda(x)$ on $\F_{2^m}$. For a nonzero
element $a\in\F_{2^m}$ we have 
\begin{align*}
H(x+a)+H(x)+H(a)&=a^{2^{r}}\Lambda(x)+\Lambda(a)x^{2^{r}}\\
                &=a^{2^{r}}(x^{2^{3r}}+\beta x^{2^{r}})+(a^{2^{3r}}+\beta a^{2^{r}})x^{2^{r}}\\
                &= a^{2^{r}}x^{2^{3r}}+a^{2^{3r}}x^{2^{r}}.
\end{align*}
That is, $H(x+a)+H(x)+H(a)=0$ if and only if  $ax^{2^{2r}}+a^{2^{2r}}x=0$, i.e., $x=0$ or $x^{2^{2r}-1}=a^{2^{2r}-1}$. Since $\gcd (2r,m)=2$, it has exactly $4$ solutions. Hence, $\delta_H(a,b)=0$ or $4$.
With Theorem \ref{delta} our claim follows.

With Theorem \ref{nobo} and Theorem \ref{Wosp}, we have to show that $H$ has only bent and semibent 
component functions.
 Let $H_\gamma=\mathrm{Tr}_m(\gamma H(x))$ be the component function corresponding to 
$\gamma\in\F_{2^m}^*$. Then we have
\begin{align*}
H_\gamma(x+a)+H_\gamma(x)+H_\gamma(a)&=\mathrm{Tr}_m(\gamma (a^{2^{r}}x^{2^{3r}}+a^{2^{3r}}x^{2^{r}}))\\
                &=\mathrm{Tr}_m(\gamma a^{2^{r}}x^{2^{3r}}+\gamma a^{2^{3r}}x^{2^{r}})\\
                &=\mathrm{Tr}_m((\gamma a^{2^{r}}+(\gamma a^{2^{3r}})^{2^{2r}})x^{2^{3r}}).
\end{align*}
Let $\alpha \in \mathbb{F}_{2^m}$ such that $\gamma =\alpha^{2^{r}}$.
Hence, $a$ is in the linear space of $H_\gamma$ if and only if $\alpha a+\alpha^{2^{2r}} a^{2^{4r}}=0$. 
That is, $a=0$ or $a^{2^{4r}-1}=\alpha^{-1(2^{2r}-1)}$. Note that the assumption $m\equiv 2 \mod 4$ 
implies that $\gcd (4r,m)=2$, i.e., $\alpha a+\alpha^{2^{2r}} a^{2^{4r}}=0$ has only the trivial solution
$a=0$, or exactly $4$ solutions. In particular, any component function of $H$ is bent or semibent.
We now may show directly, that we always have a nontrivial solution. On the other hand, since we showed that
$\delta_F(a,b) \in \{0,4\}$ for all $a\in\F_{2^n}\setminus\F_{2^m}$, we immediately obtain $|S_m| = 0$
from Corollary \ref{4-number}.
\hfill$\Box$ \\[.5em]

We proceed with a more detailed analysis of Walsh spectrum and differential spectrum and their connections,
for functions of the form $x^{2^r}\Tnm(\Lambda(x))$ for some linearized permutation $\Lambda$ on $\F_{2^m}$.
For a function $F$ on $\F_{2^n}$ and an element $a\in\F_{2^n}$ let us define $\mu_F(a,i)$ as
\[ \mu_F(a,i)=|\{b\in\fn: \delta_F(a,b)=i \}|, \]
where, as defined above, $\delta_F(a,b)=|\{x\in\fn: \ F(x+a)+F(x)=b\}|$.
We will use the following lemma for arbitrary plateaued functions on $\F_{2^n}$.
%
%
\begin{lemma}
\label{diffspec}
For an even integer $n$, let $F:\fn\to\fn$ be a plateaued, differentially $\delta$-uniform vectorial function. 
Let $|S_j|$ be the number of $j$-plateaued (nonzero) component functions. Then for any integer $k$ we have
\[
\sum_{j=0}^{k-1}\left(2^{k}-2^{j}\right)|S_j|\geq 2^k(2^n-1)-\frac{1}{2^n}\sum_{i=0}^{\delta /2} (2i)^2\sum_{a\in\fn}\mu_F(a,2i)
\]
with equality if and only if the maximal plateauedness level of $F$ is at most $k$, i.e., $F$ does not have
an $s$-plateaued component function for any $s > k$. 
\end{lemma}
\begin{proof}
For $a \in \F_{2^n}^*$ we have
\begin{align*}
\sum_{v\in\fn}\mathcal{W}_{D_aF_v}(0)^2 & = 
\sum_{v,x,y\in\fn}(-1)^{\Tr(v (F(x+a)+F(x)+F(y+a)+F(y)))}\\
  &=2^n|\{(x,y)\in\fn^2: \ D_aF(x)=D_aF(y)\}|\\
  &=2^n\sum_{i=0}^{\delta /2} (2i)^2\mu_F(a,2i)  .
\end{align*}
Summing over $a\in\fn^*$ and using $\mathcal{W}_{D_0F_v}(0)^2=
\mathcal{W}_{D_aF_0}(0)^2=2^{2n}$, we have
\[
 \sum_{v\in\fn^*}\nu(F_v)=2^{n}\sum_{i=0}^{\delta /2} (2i)^2\sum_{a\in\fn}\mu_F(a,2i).
\]
For $v\in S_j$ we have $\nu(F_v)=2^{2n+j}$, thus
\begin{align*}
 \sum_{v\in\fn^*}\nu(F_v)=\sum_{j=0}^{k-1}2^{2n+j}|S_j|+2^{2n+k}N
\end{align*}
for some integer $N$ with
\[
 \sum_{j=0}^{k-1}|S_j|+N\geq 2^n-1.
\]
Note that equality holds if and only if the maximal plateauedness level of $F$ is at most $k$, in which case
we have $N=|S_k|$. Then
\[
  \sum_{v\in\fn^*}\nu(F_v)\geq 2^{2n+k}(2^n-1) - \sum_{j=0}^{k-1}\left(2^{2n+k}-2^{2n+j}\right)|S_j|
\]
and the result follows.
\end{proof}

Note that we can summarize Lemma \ref{diffspec} and Proposition \ref{pro:S_i} in Section \ref{sec2} as follows.
\begin{corollary}
For a plateaued differentially $\delta$-uniform function $F:\fn\to\fn$ with nonlinearity 
$2^{n-1}-\frac{1}{2}2^{\frac{n+k}{2}}$
we have
\[ 2^{k}(2^n-1)-2^{-n}\sum_{i=0}^{\delta /2} (2i)^2\sum_{a\in\fn}\mu_F(a,2i) = \sum_{i=0}^{k-1}\left(2^k-2^i 
\right)|S_i|\leq (2^{k}-2)(2^n-1). \]
\end{corollary}

Observe that by Theorem \ref{delta}, for $F(x) = x^{2^r}\Tnm(\Lambda(x))$ and $a\in\F_{2^n}\setminus\F_{2^m}$
with $\delta(a,b) \in \{2^s,0\}$ we have $\delta(a_1,b) \in \{2^s,0\}$ if $\Tnm(a_1) = \Tnm(a)$. Hence we then 
also have $\mu_F(a_1,i) = \mu_F(a,i)$.
Therefore, for a function $F(x) = x^{2^r}\Tnm(\Lambda(x))$ and $z\in\F_{2^m}^*$ we define 
\[ \tilde{\mu}_F(z,i) := \mu_F(a,i)\;\;\mbox{if}\;\; \Tnm(a) = z. \]
Note that for all $a\in\F_{2^n}\setminus\F_{2^m}$ we have $\delta_F(a,b) \in \{0,2^s\}$ {for some integer
$s$ depending only on $a$}. Hence by the definition of $\mu_F(a,i)$ and $\tilde{\mu}_F(z,i)$, for $\Tnm(a) = z \in\F_{2^m}$
we have 
\begin{equation}\label{eq:tilde}
{
\tilde{\mu}_F(z,2^s) = 2^{n-s} \quad \text{and} \quad  \tilde{\mu}_F(z,i) = 0 \quad \text{if } i\ne 2^s.
}
\end{equation}
\begin{lemma}
\label{onelemma}
Let $F(x) = x^{2^r}\Tnm(\Lambda(x))$ for a linearized permutation $\Lambda \in\F_{2^m}[x]$, and let 
$\mathcal{N}_F = 2^{n-1}-\frac{1}{2}2^{(n+k)/2}$ be the nonlinearity of $F$. If the differential uniformity 
of $H(x) = x^{2^r}\Lambda(x)$ is $\delta_H = 2^\sigma$, then
\begin{equation}
\label{mu=Sj}
\frac{1}{2^n}\sum_{s=1}^\sigma 2^{m+2s} \sum_{z\in\F_{2^m}^*}\tilde{\mu}_F(z,2^s) = (2^m-1)2^k - \sum_{j=2}^{k-2}(2^k-2^j)|S_j|. 
\end{equation}
\end{lemma}
{\it Proof.}
With $|S_0| = 2^n-2^m$, from Lemma \ref{diffspec} we get 
\[ \frac{1}{2^n}\sum_{i=0}^{\delta_F /2} (2i)^2\sum_{a\in\F_{2^n}}\mu_F(a,2i) = (2^m-1)(2^k+2^m) - \sum_{j=2}^{k-2}\left(2^{k}-2^{j}\right)|S_j| . \]
We use that $\mu_F(a,i) = 0$ if $i$ is not a power of $2$, that $\delta_F(a,b)\in\{0,2^m\}$, hence $\mu_F(a,2^m) = 2^{n-m} = 2^m$ if $a\in\F_{2^m}^*$. By the connections between the differential spectra of $F$ and $H$ as given
in Theorem \ref{delta}, we obtain 
\[ \frac{1}{2^n}\sum_{s=1}^{\sigma} 2^{2s}\sum_{a\in\F_{2^n}\setminus\F_{2^m}}\mu_F(a,2^s) + (2^m-1)2^m = (2^m-1)(2^k+2^m) - \sum_{j=2}^{k-2}\left(2^{k}-2^{j}\right)|S_j|  . \]
As $\mu_F(a,2^s) = \tilde{\mu}_F(z,2^s)$ if $\Tnm(a) = z$, this yields
\[ \frac{1}{2^n}\sum_{s=1}^{\sigma} 2^{m+2s}\sum_{z\in\F_{2^m}^*}\tilde{\mu}_F(z,2^s) = (2^m-1)2^k - \sum_{j=2}^{k-2}\left(2^{k}-2^{j}\right)|S_j|  . \]
\hfill$\Box$
\begin{example}
For the function $x^{2^r}\Tnm(x)$, $m=n/2$ odd, and $\gcd(r,m) = 1$, we have $k = m+1$, $|S_{m+1}| = 2^m-1$ ($|S_j| = 0$ if $j\ne 0,m+1$). Since $\delta(a,b) = 2$ or $0$
for $a\not\in\F_{2^m}$ (and all $b$), we have $\mu_F(a,2) = \tilde{\mu}_F(z,2) = 2^{n-1}$ for all $a\in \F_{2^n}\setminus\F_{2^m}$ respectively $z\in\F_{2^m}^*$
($\tilde{\mu}_F(z,2^s) = 0$ for all $z\in\F_{2^m}^*$ if $s\ne 1$). Inserting in $(\ref{mu=Sj})$ we obtain $2^{m+1}(2^m-1)$ on both sides of the equation.
\end{example}

By the above discussions, also for $m$ even we are left behind with solely one type of functions
if we demand that $\delta_F(a,b) \in \{0,2\}$ for all $a\in\F_{2^n}\setminus\F_{2^m}$.
The corresponding function $H(x)$ is then (up to a multiplication by a constant) the Gold function $x^{2^r+1}$,
$\gcd(r,m)= 1$, the only APN function of the form $x\Lambda(x)$, and the only function of the form 
$x\Lambda(x)$ with $2(2^m-1)/3$ bent components and $(2^m-1)/3$ semibent components
($\Lambda$ is some linearized polynomial). As our experimental results indicate there are plenty of
other functions with maximal nonlinearity, of course with $|S_m| < 2(2^m-1)/3$. Many of them
satisfy $\delta_F(a,b)\in\{0,2,4\}$ for all $a\in\F_{2^n}\setminus\F_{2^m}$, in which case $|S_m|$
fixes also the number of $4$'s in the differential spectrum.
But in general, $F$ having the maximal nonlinearity, does not imply that $\delta_F(a,b)\in\{0,2,4\}$ for all 
$a\in\F_{2^n}\setminus\F_{2^m}$, as Example \ref{Ex-8U} in Section \ref{sec4} shows.

%
%

For the case that $|S_m| < 2(2^m-1)/3$ (which applies to all but one type of functions with 
maximal nonlinearity), using the Hasse-Weil bound, we will show an upper bound for $|S_m|$ in terms of 
the polynomial degree of $H(x) =x^{2^r}\Lambda(x)$.

In the proof of the next theorem we will use the following notation for a given function $F(x) = x^{2^r}\Tnm(\Lambda(x))$:
\[ A_s := \{z\in\F_{2^m}^*\,:\, \tilde{\mu}_F(z,2^s)\ne 0\} = 
\{z\in\F_{2^m}^*\,:\, \tilde{\mu}_F(z,2^s) = 2^{n-s}\}, \]
{where the equality follows from \eqref{eq:tilde}.}
We first show a lemma which relates $|S_m|$ with the cardinality of the above defined sets $A_s$.
\begin{lemma}
Let $n=2m$, $m$ even, and suppose that for the linearized permutation $\Lambda\in\F_{2^m}[x]$ the
function $F(x) = x^{2^r}\Tnm(\Lambda(x))$ on $\F_{2^n}$ has the maximal possible nonlinearity
$\mathcal{N}_F = 2^{n-1} - 2^{\frac{n+m}{2}}$. Then we have
\begin{align}\label{eq:upperr}
3|S_m|=\sum_{s=1}^\sigma |A_s|(4-2^{s} ) ,
\end{align}
where $\delta_H = 2^\sigma$ is the differential uniformity of $H(x) = x^{2^r}\Lambda(x)$.
\end{lemma}
{\it Proof.}
Since all non-bent components are $m$-plateaued or $(m+2)$-plateaued, by Lemma \ref{onelemma} we have
\[ \frac{1}{2^n}\sum_{s=1}^\sigma |A_s| 2^{m+2s}2^{n-s} = 2^{m+2}(2^m-1) - 3\cdot 2^m|S_m|. \]
Dividing by $2^m$ we obtain
\[ \sum_{s=1}^\sigma |A_s|2^{s} = 4(2^m-1)-3|S_m|. \]
Then together with $\sum_{s=1}^\sigma |A_s| = 2^m-1$, this yields the desired result.
\hfill$\Box$ \\[.5em]
\begin{remark}
In particular, if $F(x) = x^{2^r}\Tnm(\Lambda(x))$ has maximal possible nonlinearity and 
$\delta_F(a,b) \in \{0,2,4\}$ for all $a\in\F_{2^n}\setminus\F_{2^m}$, then by Equation \eqref{eq:upperr} we have $3|S_m| = 2|A_1|$.
\end{remark}

Let $n= 2m$, $m$ even, and for a linearized permutation let $F(x) = x^{2^r}\Tnm(\Lambda(x))$ 
be a function attaining the maximum nonlinearity. If $\mathrm{deg}(x^{2^r}\Lambda(x))<2^{m/4}$, 
then using the Hasse-Weil bound, we can show an upper bound for the number $|S_m|$ of $m$-plateaued 
components.
\begin{theorem} 
Let $n= 2m$, $m$ even. Suppose that for a linearized permutation $\Lambda \in\F_{2^m}[x]$ the function 
$F(x) = x^{2^r}\Tnm(\Lambda(x))$ has the maximal possible nonlinearity. {Let $\delta_H=2^\sigma$ be the differential uniformity of $H(x)=x^{2^r}\Lambda(x)$.}
If $\sigma>1$ then
\begin{align*}
|S_m|<\frac{2(2^m-1)}{3}-\frac{2^{m/2}(2^{m/2}-\mathrm{deg}(x^{2^r}\Lambda(x))^2)}{3}.
\end{align*}
\end{theorem}
\begin{proof}
We observe that the numbers $|A_s|$ are determined by the solution space of the equation 
$G(x,z):=x^{2^r}\Lambda(z) + z^{2^r}\Lambda(x)=0$. That is, we have
\begin{align*}
|A_1|2^1+|A_2|2^2+\cdots+ |A_\sigma|2^\sigma=|\lbrace (x,z)\in \F_{2^m}\times \F_{2^m}^* \; : \; G(x,z)=0 \rbrace|.
\end{align*}
Note that for any $z\in \F_{2^m}^*$ $(0,z)$ and $(z,z)$ always the solution of the equation. Hence, 
\begin{align*}
|A_2|(2^2-2)+\cdots+ |A_\sigma|(2^\sigma -2)=|\lbrace (x,z)\in \F_{2^m}^*\times \F_{2^m}^* \; | \; G(x,z)=0  \;\, \text{and}\;\; x\neq z\rbrace|.
\end{align*}
Without loss of generality suppose that $\Lambda(x)$ is separable; and hence $G(x,z)$ is separable. Otherwise, there exists a separable polynomial $\tilde{G}(x,z)$  such that $G(x,z)=\tilde{G}(x,z)^{2^\ell}$ for some positive integer $\ell$. Then we consider $\tilde{G}$ instead of $G$. 

{Write} 
$\Lambda(x)=\sum c_ix^{2^i} $ with $c_0\neq 0$ and $\mathrm{deg}(\Lambda(x))=2^t$. Then 
\begin{align*}
\frac{\partial G(x,z)}{\partial x}= c_0z^{2^r}  \quad \text{and} \quad
 \frac{\partial G(x,z)}{\partial z}= c_0x^{2^r}.
\end{align*}
Hence, $(0,0)$ is the only singular affine point of the curve $G$. Without loss of generality we suppose that $2^r>\mathrm{deg}(\Lambda(x))=2^t$. Since $\sigma>1$ there exists a rational point $P=(x,z)$ on $G$ such that $xz\neq 0$ and $x\neq z$. In particular, $P$ does not lie on the components $x$, $z$, $x+z$ and it is non-singular. Hence, by \cite[Lemma 2.1]{difference} the absolutely irreducible component $C$ of $G$ passing through $P$ is defined over $\F_{2^m}$. Note that $\mathrm{deg}(C)\leq \mathrm{deg}(G/xz(x+z))=2^r+2^t-3$, and that, since $G$ is a separable polynomial, $C$ is a component of $G$ different from $x$, $z$ and $x+z$. Then, by the Hasse-Weil bound, the number of affine rational points $N(C)$ of $C$ satisfies
\begin{align*}
N(C)&\geq 2^m+1-2^{m/2}(2^r+2^t-4)(2^r+2^t-5)-(2^{r-t}+1)\\
&=2^m-(2^{m/2}(\mathrm{deg}(x^{2^r}\Lambda(x))-4)(\mathrm{deg}(x^{2^r}\Lambda(x))-5)+2^{r-t}),
\end{align*}
where $2^{r-t}+1$ is the number of distinct points of $G$ at infinity; namely, $(1:0:0)$, $(0:1:0)$, $(1:\zeta:0)$ with $\zeta^{2^{r-t}-1}=1$. Note that the only rational affine point of $G/xz(x+z)$ is $(0,0)$, which also lies on $x+z$. Hence by Bezout's Theorem, there exist at most $\mathrm{deg}(C)$ affine rational points that lie on $C\cap \{x=z\}$. That is, for $d=\mathrm{deg}(x^{2^r}\Lambda(x))$ we have
\begin{align}\label{eq:sum}
\sum_{i=1}^{\sigma}|A_i|(2^i-2)&\geq 
2^m-(2^{m/2}(d-4)(d-5)+2^{r-t}+d-3) \nonumber \\
& > 2^m-2^{m/2}d^2.
\end{align}
By Equation \eqref{eq:upperr}, we have
\begin{align}\label{eq:summ}
\sum_{i=1}^{\sigma}|A_i|(2^i-2)=2\sum_{i=1}^{\sigma}|A_i|-3|S_m|=2(2^m-1)-3|S_m|.
\end{align}
Then by {Equations} \eqref{eq:sum} and \eqref{eq:summ}, we obtain the desired result.
\end{proof}
%

%
%

\section{Experimental results and remarks}
\label{sec4}

We finish the article with some experimental results and their interpretation.
We present Walsh spectrum and differential spectrum for all polynomials of the form
$F(x) = x^{2^r}\T^n_m(\Lambda(x))$ over $\F_{2^n}$, $\Lambda\in\F_{2^m}[x]$ is a linearized 
permutation polynomial of $\F_{2^m}$, in Table \ref{tab1} for $n = 2m = 8$, and in Table \ref{tab2} for 
$n = 2m = 10$. We remark that Theorem 2.1 in \cite{zhou} presents an explicit representation of all 
linearized permutation polynomials of a finite field $\F_{2^m}$.

Since up to EA-equivalence, all functions of our form can be represented with some fixed integer $r$ and 
and a normalized linearized permutation $\Lambda$, in our calculations we fix $r=1$ and run through all
normalized linearized permutation polynomials in $\F_{2^4}[x]$ respectively $\F_{2^5}[x]$.

We explain the notation in Tables \ref{tab1} and \ref{tab2}: For example, in Category 4 (Cat. 4) in Table \ref{tab1}
we have $15$ functions, for which the Walsh spectrum is $(4^{12}, 6^2, 8^1)$, i.e., besides from the bent 
component functions we have $12$ $4$-plateaued ($|S_4| = 12$), two $6$-plateaued ($|S_6| = 2$) 
component functions, and one $8$-plateaued ($|S_8| = 1$) component function. For all functions $F$ in
Category 4 there are $192$ elements $a\in \F_{2^8}\setminus\F_{2^4}$ for which $\delta_F(a,b) \in \{0,2\}$, 
and there are $48$ elements $a\in \F_{2^8}\setminus\F_{2^4}$ for which $\delta_F(a,b) \in \{0,4\}$.
\begin{tab}
\label{tab1}
Walsh spectrum and differential spectrum for functions $x^{2^r}\T^8_4(\Lambda(x))$ on $\F_{2^8}$ \\[.5em]
\begin{center}
\begin{tabular}{|c|c|c|c|} \hline
Cat. & \# & Walsh spectrum & differential spectrum  \\[.3em] \hline \hline
1 & $2$ & $(4^{10}, 6^5)$ & $\{0, 2\}_{240}$ \\[.3em] \hline
2 & $180$ & $(4^8, 6^7)$ & $\{0, 2\}_{192}$, $\{0, 4 \}_{48}$ \\[.3em] \hline
3 & $750$ & $(4^6, 6^9)$ & $\{0, 2\}_{144}$, $\{0, 4 \}_{96}$ \\[.3em] \hline
4 & $15$ & $(4^{12}, 6^2, 8^1)$ & $\{0, 2\}_{192}$, $\{0, 4 \}_{48}$ \\[.3em] \hline
5 & $280$ & $(4^8, 6^6, 8^1)$ & $\{0, 2\}_{96}$, $\{0, 4 \}_{144}$ \\[.3em] \hline
6 & $11$ & $(4^{12}, 8^3)$ & $\{0, 4 \}_{240}$ \\[.3em] \hline
7 & $105$ & $(6^{14}, 8^1)$ & $\{0, 2\}_{128}$, $\{0, 8 \}_{112}$ \\[.3em] \hline
8 & $1$ & $(8^{15})$ & $\{0, 16 \}_{240}$ \\[.3em] \hline
\end{tabular}
\end{center}
\end{tab}
As seen from Table \ref{tab1}, in dimension $8$, our functions with the maximal possible number of
bent components separate into at least $8$ ($7$ nontrivial) CCZ-equivalence classes. 
Category $8$ contains the trivial example (equivalent to $x^{2^4+1}$). The functions in Categories
$1,2$ and $3$ exhibit the largest possible nonlinearity, see Theorem \ref{nobo}. By Corollary \ref{allbest}, we
know that the first category exactly contains the functions $x^2\T^8_4(x)$ and $x^{2^3}\T^8_4(x)$, which is equivalent
 to $x^2\T^8_4(x^{2^2})$.  By Corollary \ref{4-number}, the differential spectrum in Categories 1,2,3 is determined
by the Walsh spectrum. In fact in dimension $8$ all functions with the same Walsh spectrum also have the same
differential spectrum (whereas we see the same differential spectrum but different Walsh spectrum for
Categories 2 and 4). The same apples to dimension $n=10$, see Table \ref{tab2} below.
But this is not true in general as the following example in dimension $12$ shows.
     \begin{example}
     Let $m=6$, $n=12$, $r=1$, $<\gamma> = \mathbb{F}_{2^m}^{*}$.  
     Let\\ $\Lambda_1(x)=\gamma^{52}x^{32} + \gamma^{40}x^{16} + \gamma^{35}x^8 + \gamma^{52}x^4 + \gamma^{58}x^2$,\\ $\Lambda_2(x)=\gamma^{10}x^{32} + \gamma^{49}x^{16} + \gamma^{26}x^8 + \gamma^{14}x^4 + \gamma^{40}x^2 + \gamma^{30}x$.\\
     Consider $F_1(x)=x^{2^r}\hbox{\rm{Tr}}_{m}^{n}(\Lambda_1(x))$ and $F_2(x)=x^{2^r}\hbox{\rm{Tr}}_{m}^{n}
(\Lambda_2(x))$.\\
Then $F_1$ and $F_2$ have the Walsh spectrum $(6^{35},8^{26},10^{2})$. For the differential spectrum of
$F_1$ we have $\{0,2\}_{2624}$, $\{0,4\}_{960}$, $\{0,8\}_{448}$, whereas the differential spectrum of $F_2$ is 
described by $\{0,2\}_{1728}$, $\{0,4\}_{2304}$.
\end{example}
As seen in the next example, when $m$ is even, the maximal possible nonlinearity does not imply 
$\delta_F(a,b) \le 4$ for all $a\in\F_{2^n}\setminus\F_{2^m}$.
\begin{example}
\label{Ex-8U}
Let $m=6$, $r=1$, $<\gamma>=\mathbb{F}_{2^m}^{*}$, $\Lambda(x)=x^{32} + \gamma^{18}x^{16} + \gamma^{40}x^8 + \gamma^{49}x^4 + \gamma^7x^2 + \gamma^{57}x$. Consider $H(x)=x^{2^r}\Lambda(x)$. Then $H$ has $16$ bent and $47$ semibent components and the differential spectrum of $H$ is $\{0, 2\}_{38}$,  $\{0, 4\}_{18}$,   $\{0, 8\}_{7}$.
\end{example}
\begin{tab}
\label{tab2}
Walsh spectrum and differential spectrum for functions $x^{2^r}\T^{10}_5(\Lambda (x))$ on $\F_{2^{10}}$ \\[.5em]
\begin{center}
\begin{tabular}{|c|c|c|c|} \hline
Cat. & \# & Walsh spectrum & differential spectrum  \\[.3em] \hline \hline
1 & $4$ & $(6^{31})$ &  $\{0, 2\}_{992}$ \\[.3em] \hline
2 & $4650$ & $(6^{29}, 8^2)$ & $\{0, 2\}_{800}$, $\{0, 4 \}_{192}$ \\[.3em] \hline
3 & $43400$ & $(6^{28}, 8^3)$ & $\{0, 2\}_{704}$, $\{0, 4 \}_{288}$ \\[.3em] \hline
4 & $116000$ & $(6^{27}, 8^4)$ &  $\{0, 2\}_{608}$, $\{0, 4 \}_{384}$ \\[.3em] \hline
5 & $77748$ & $(6^{26}, 8^5)$ &  $\{0, 2\}_{512}$, $\{0, 4 \}_{480}$ \\[.3em] \hline
6 & $28210$ & $(6^{25}, 8^6)$ & $\{0, 2\}_{416}, \{0, 4 \}_{576}$ \\[.3em] \hline
7 & $2170$ & $(6^{24}, 8^7)$ &  $\{0, 2\}_{768}$, $\{0, 8\}_{224}$ \\[.3em] \hline
8 & $4092$ & $(6^{30}, 10^1)$ &  $\{0, 2\}$, $\{0, 16 \}_{240}$ \\[.3em] \hline
9 & $9300$ & $(6^{24}, 8^6, 10^1)$ & $\{0, 2\}_{384}$, $\{0, 4 \}_{384}, \{0, 8\}_{224}$ \\[.3em] \hline
10 & $465$ & $(8^{30}, 10^1)$ & $\{0, 2\}_{512}$, $\{0, 16 \}_{480}$ \\[.3em] \hline
11 & $40920$ & $(8^9, 10^{22})$ & $\{0, 2\}_{576}$, $\{0, 4 \}_{192}, \{0, 8\}_{224}$ \\[.3em] \hline
12 & $1$ & $(10^{31})$ &  $\{0, 32 \}_{992}$ \\[.3em] \hline
\end{tabular}
\end{center}
\end{tab}
%
As seen in Table \ref{tab2}, in dimension $10$, our functions separate into at least $12$
CCZ-equivalence classes, which also indicates that there is a large variety of (nontrivial) CCZ-inequivalent functions 
of the form $F(x) = x^{2^r}\Tnm(\Lambda(x))$ with the largest possible number of bent components.

Observe that besides from the first category, which by Corollary \ref{allbest} consists of functions equivalent to
$x^{2^r}\Tnm(x)$, $1\le r\le 4$, there are $6$ more categories with not any linear function as component. 
As also indicated in Remark \ref{muva}, such functions are particularly interesting. Note also that a linear component 
causes the worst possible nonlinearity. In this connection, an interesting question is the following: \\[.3em]
{\it Question 1:}
Find more examples of functions with the maximal number of bent components with nontrivial bent
complement, i.e., functions which are not CCZ-equivalent (possibly not quadratic) to any function 
in this article. Most interesting functions are the ones with high nonlinearity or at least without any linear 
component. \\[.3em]
All functions in dimension $n$ with the maximal number of bent components contain a large variety of
vectorial bent functions from dimension $n$ into dimension $m$. The obvious ones in the functions in this
article are vectorial Maiorana-McFarland functions, but there may also be other ones, see Question 1 
in \cite{ppmb}. \\[.3em] 
{\it Question 2:}
Which vectorial bent functions from dimension $n$ to dimension $n/2 = m$ permit a nontrivial extension
to a vectorial function in dimension $n$ with $2^n-2^m$ bent components? Most interesting are extensions
without any affine component functions, or even with largest possible nonlinearity. How many such extensions 
can exist for a given vectorial bent function? \\[.3em]
{\it Question 3:}
How many such vectorial bent functions can two (of our) functions on $\F_{2^n}$ with the maximal number
of bent components, and possibly both with maximal nonlinearity, have in common? \\[.3em]
Several more, interesting questions on functions with the maximal number of bent components can be found
in \cite{ppmb}. We remark that Question 4 and parts of Question 8 are solved in \cite{mztz}, where it is
shown that having the maximal number of bent components is a CCZ-invariant, and that a plateaued function with the
maximal number of bent components cannot be APN. Parts of Question 5 we solved in Section \ref{sec2}.
In these connections we ask a further question. \\[.3em]
{\it Question 4:}
Is there a function with the maximal number of bent components which is not CCZ-equivalent to a plateaued 
function? \\[.3em]
Note that with a negative answer to Question 4, Theorem \ref{nobo} would completely solve the question on 
the maximal nonlinearity for functions with the maximal number of bent components.

We finish the article with some remarks on Question 3 in \cite{ppmb}, which strengthens the assumption of the 
authors of \cite{ppmb}, that $x^{2^m+1}$ is the only monomial with the maximal number of bent components.
Note that this is confirmed for all $n\leq 20$, since all monomial bent functions in dimension $n \le 20$ are known.
\begin{fact}
\label{onlymo}
Let $F(x)=x^{d}$ be a monomial function having the maximum number bent components. Then $d=t(2^m+1)$ 
for some integer $t$ relatively prime to $2^m-1$.
\end{fact}
\begin{proof}
The component function $\mathrm{Tr}_n(\alpha F(x))$ is bent if and only if $\mathrm{Tr}_n(\alpha ((x+a)^d+x^d))$ is balanced for all non-zero $a\in \F_{2^n} $. This holds if and only if $\mathrm{Tr}_n(\alpha a^d ((x+1)^d+x^d))$ is balanced for all non-zero $a\in \F_{2^n} $. That is, $\mathrm{Tr}_n(\alpha F(x))$ is bent if and only if  $\mathrm{Tr}_n(\tilde{\alpha} F(x))$ is bent for all non-zero $\tilde{\alpha} \in \alpha \mathrm{Im}(x^d)$. Therefore, the number of bent components is divisible by $|\mathrm{Im}(x^d)\setminus \lbrace 0 \rbrace|=(2^n-1)/\gcd(2^n-1,d)$. As the maximum number of bent components is $2^m(2^m-1)$, for some positive integer $s$ we have
\begin{align*}
s\frac{2^n-1}{\gcd(2^n-1,d)}=2^m(2^m-1) ,
\end{align*} 
i.e., $\gcd(2^n-1,d)2^m=s(2^m+1)$. Therefore, $\gcd(2^n-1,d)$ is divisible by $2^m+1$. Write 
$d=t(2^m+1)$ for some positive integer $t$. Now we show that $t$ is relatively prime to $2^m-1$. Note that 
$\mathrm{Tr}_n(\alpha F(x))$ is the zero function for all $\alpha \in\F_{2^m}$, i.e., $\mathrm{Tr}_n(\alpha F(x))$ 
is bent for all $\alpha\in \F_{2^n}\setminus\F_{2^m}$. In particular, $W_F(\alpha,0)=\pm 2^m$. 
Set $\eta=\Tnm(\alpha)$, then we have the following equalities.
\begin{align} \label{eq:walsh}
W_F(\alpha,0)&= \sum_{x\in \F_{2^n}}(-1)^{\mathrm{Tr}_n(\alpha (x^{2^m+1})^{t})}
=1+(2^m+1)\sum_{y\in \F_{2^m}\setminus \lbrace 0 \rbrace}(-1)^{\mathrm{Tr}_m(\eta y^{t})} \nonumber \\
&=-2^m+ (2^m+1)\sum_{y\in \F_{2^m}}(-1)^{\mathrm{Tr}_m(\eta y^{t})}=\pm 2^m.
\end{align}
We conclude that $\sum_{y\in \F_{2^m}}(-1)^{\mathrm{Tr}_m(\eta y^{t})}$ is divisible by $2^m$. Then the fact that $-2^m \leq \sum_{y\in \F_{2^m}}(-1)^{\mathrm{Tr}_m(\eta y^{t})} \leq -2^m$ implies that $\sum_{y\in \F_{2^m}}(-1)^{\mathrm{Tr}_m(\eta y^{t})} \in \lbrace -2^m,0,2^m \rbrace $. By Equation \eqref{eq:walsh}, we conclude that 
$\sum_{y\in \F_{2^m}}(-1)^{\mathrm{Tr}_m(\eta y^{t})}=0$. Set $I=\mathrm{Im}(y^t)\setminus \lbrace 0 \rbrace $. Then 
$\sum_{y\in \F_{2^m}}(-1)^{\mathrm{Tr}_m(\eta y^{t})}=0$ implies that 
$\sum_{y\in \F_{2^m}\setminus \lbrace 0 \rbrace}(-1)^{\mathrm{Tr}_m(\eta y^{t})}=-1$, by which
\begin{align*}
\sum_{y\in \F_{2^m}\setminus \lbrace 0 \rbrace}(-1)^{\mathrm{Tr}_m(\eta y^{t})}=\gcd(t,2^m-1)
\sum_{z\in I}(-1)^{\mathrm{Tr}_m(\eta z^{t/\gcd(t,2^m-1)})}=-1 .
\end{align*}
Hence $\gcd(t,2^m-1)$ divides $1$, i.e., $\gcd(t,2^m-1)=1$.
\end{proof}

Note that for showing that $x^{2^m+1}$ is (up to equivalence) the only monomial with the maximal number
of bent components, we have to show that $t$ in Fact \ref{onlymo} must be a power of $2$. Hence it is sufficient 
to show that if $t$ is odd, then we must have $t=1$, which seems not to be an easy task.

\section*{Acknowledgement}

W.M. is supported by the FWF Project P 30966; L.M. is supported by the FWF Project P 31762. \\[.3em]
Wilfried Meidl wishes to thank Dr. Anbar for the hospitality during a research visit at Sabanc{\i} University,
21.09.2020 -- 06.10.2020.

\end{document}